\documentclass[11pt,letterpaper]{article}
\usepackage{amssymb,amsmath,theorem,tikz}

\newcommand{\R}{\mathbb{R}}

\newcommand{\RGG}{{\cal{G}}(n,r)}

\newcommand{\kn}{\omega(n)}
\newcommand{\cN}{{\overline{N}}}

\newcommand{\Ncal}[0]{\ensuremath{{\mathcal N}}}

\newcommand{\eR}[0]{\ensuremath{ \mathbb R}}

\newcommand{\eN}[0]{\ensuremath{ \mathbb N}}

\newcommand{\ctil}[0]{\tilde{c}}

\newcommand{\norm}[1]{\ensuremath{\|#1\|}}

\newcommand{\Pee}[0]{\ensuremath{{\mathbb P}}}
\newcommand{\Ee}[0]{\ensuremath{{\mathbb E}}}
\newcommand{\Var}[0]{\ensuremath{{\mbox{Var}}}}

\newcommand{\isd}[0]{\hspace{.2ex} \raisebox{-.1ex}{$=$} \hspace{-1.5ex} 
\raisebox{1ex}{{$\scriptstyle d$}} \hspace{.8ex} }

\DeclareMathOperator{\area}{area}

\DeclareMathOperator{\diam}{diam}

\DeclareMathOperator{\dd}{d}
\DeclareMathOperator{\Po}{Po}

\newcommand{\conM}[0]{{\bf(M)}}
\newcommand{\nb}[0]{\ensuremath{\Ncal_{c}}}

\newenvironment{my_itemize}{
\begin{itemize}
  \setlength{\itemsep}{1pt}
  \setlength{\parskip}{1pt}
  \setlength{\parsep}{0pt}}{\end{itemize}
}

\newif\ifnotesw\noteswtrue

\newtheorem{theorem}{Theorem}[section]

\newtheorem{lemma}[theorem]{Lemma}
\newtheorem{cor}[theorem]{Corollary}
\theorembodyfont{\rmfamily}

\newcommand{\bpf}[1][Proof.]{\smallskip\noindent{\it #1} }
\newcommand{\bpfof}[1]{\smallskip\noindent{\emph{Proof of #1}.}}
\newcommand{\qed}{\nolinebreak\mbox{\hspace{5 true pt}%
  \rule[-0.85 true pt]{3.9 true pt}{8.1 true pt}}}

\newcommand{\epf}{\qed \medskip}

\newenvironment{proof}{\vspace{1ex}\noindent{\bf Proof:}}{\hspace*{\fill}$\blacksquare$\vspace{1ex}}
\newenvironment{proofof}[1]{\vspace{1ex}\noindent{\bf Proof of #1:}}{\hspace*{\fill}$\blacksquare$\vspace{1ex}}

\begin{document}

\title{Cops and Robbers on Geometric Graphs}
\author{
Andrew Beveridge\thanks{Department of Mathematics, Statistics and Computer
Science, Macalester College, Saint Paul, MN: \texttt{abeverid@macalester.edu}}\,, 
Andrzej Dudek\thanks{Department of Mathematics, Western Michigan University, Kalamazoo, MI:
\texttt{andrzej.dudek@wmich.edu}}\,,
Alan Frieze\thanks{Department of Mathematical Sciences, Carnegie Mellon University, Pittsburgh, PA: \texttt{alan@random.math.cmu.edu}. Supported in part by NSF Grant CCF1013110}\,,
Tobias M\"uller\thanks{Centrum voor Wiskunde en Informatica, Amsterdam, the Netherlands: \texttt{tobias@cwi.nl}.
Supported in part by a VENI grant from Netherlands Organization for Scientific Research (NWO)}
}

\maketitle

\begin{abstract}
Cops and robbers is a turn-based pursuit game played on a graph $G$. 
One robber is pursued by a set of cops. In each round, these agents move between vertices along the edges of the graph. 
The cop number $c(G)$ denotes the minimum number of cops required to catch the robber in finite time. 
We study the cop number of geometric graphs. For points $x_1, \ldots , x_n \in \R^2$, and $r \in \R^+$,  the vertex set of 
the geometric graph $G(x_1, \ldots , x_n; r)$ is the graph on these $n$ points, with  $x_i, x_j$ adjacent 
when $ \norm{x_i -x_j } \leq r$.   We prove that $c(G) \leq 9$ for any connected geometric graph $G$ in $\R^2$
and we give an example of a connected geometric graph with $c(G) = 3$. 
We improve on our upper bound for random geometric graphs that are sufficiently dense. 
Let $\RGG$ denote the probability space of geometric graphs with $n$ vertices chosen uniformly and independently from $[0,1]^2$. 
For $G \in \RGG$, we show that with high probability (whp),   
if $r \geq K_1 (\log n/n)^{\frac14}$, then $c(G) \leq 2$, and  if $r \geq K_2(\log n/n)^{\frac15}$, then $c(G) = 1$ 
where $K_1, K_2 > 0$ are absolute constants. 
Finally, we provide a lower bound near the connectivity regime of $\RGG$: if $r \leq K_3 \log n / \sqrt{n} $ then $c(G) > 1$ whp, where $K_3 > 0$ is an absolute constant.
\end{abstract}

\section{Introduction}

The game of cops and robbers is a full information game played on a graph $G$. The game was introduced independently by Nowakowski and Winkler \cite{Nowakowski+Winkler} and  Quilliot \cite{quilliot}. During play, one robber $R$ is pursued by a set of cops $C_1, \ldots, C_{\ell}$. Initially, the cops choose their locations on the vertex set. Next, the robber chooses his location. The cops and the robber are aware of the location of all agents during play, and the cops can coordinate their motion.  On the cop turn, each cop  moves to an adjacent vertex, or remains stationary. This is followed by the robber turn, and he moves similarly. The game continues with the players alternating turns. The cops win if they can catch the robber in finite time, meaning that some cop is colocated with the robber. The robber wins if he can evade capture indefinitely.

The original formulation \cite{Nowakowski+Winkler,quilliot}  concerned a single cop chasing the robber. These papers  characterized the structure of \emph{cop-win} graphs for which a single cop has a winning strategy. For $v \in V(G)$, the \emph{neighborhood} of $v$ is $N(v) = \{ u \in V(G) \mid (u,v) \in E(G) \}$ and the \emph {closed neighborhood} of $v$ is $\cN(v) = \{ v \} \cup N(v)$. When $\cN(u) \subseteq \cN(v)$, we say that $u$ is a \emph{pitfall}. A graph is \emph{dismantlable} if we can reduce $G$ to a single vertex by successively removing pitfalls. 

\begin{theorem}[\cite{Nowakowski+Winkler, quilliot}]
\label{thm:pitfall}
$G$ is dismantlable if and only if $c(G)=1$.
\end{theorem}

Aigner and Fromme \cite{aigner+fromme} introduced the multiple cop variant described above. For a fixed graph $G$, they defined the \emph{cop number} $c(G)$ as the minimum number of cops for which there is a winning cop strategy on $G$. Among their results, they proved the following.

\begin{theorem}[\cite{aigner+fromme}]
\label{thm:planar}
If $G$ is a connected planar graph, then $c(G)\leq 3$.
\end{theorem}

Various authors have studied the cop number of  families of graphs \cite{frankl, frankl2, mehrabian, neufeld}. Recently, significant attention has been directed towards Meyniel's conjecture (found in \cite{frankl2}) that $c(G) = O(\sqrt{n})$ for any $n$ vertex graph. The best current bound is $c(G) \leq n 2^{-(1+o(1))\sqrt{\log n}}$, obtained independently in \cite{lu+peng,scott+sudakov,FKL}. The history of Meyniel's conjecture is surveyed in \cite{baird+bonato}. For further results on vertex pursuit games on graphs, see the surveys \cite{alspach, hahn} and the monograph \cite{bonato+nowakowski}.

Herein, we study the game of cops and robbers on geometric graphs in $\R^2$. 
Given points $x_1, \ldots , x_n \in \R^2$ and $r \in \R^+$, the \emph{geometric graph} $G=G(x_1, \ldots, x_n; r)$ has vertices $V(G) = \{ 1, \ldots , n \}$ and $i j \in E(G)$ if and only if $\norm{ x_i - x_j} \leq r$. 
Geometric graphs  are widely used to model ad-hoc wireless networks \cite{gupta+kumar,xue+kumar}.
For convenience, we will consider $V(G) = \{x_1, \ldots x_n \}$, referring to ``point $x_i$'' or ``vertex $x_i$'' when this distinction is required. Our first result gives a constant upper bound on the cop number of 2-dimensional geometric graphs.

\begin{theorem}
\label{thm:nine}
If $G$ is a connected geometric graph in $\R^2$, then $c(G) \leq 9$. 
\end{theorem}
The proof of this theorem is an adaptation of the proof of Theorem \ref{thm:planar}. This adaptation  requires three cops on a geometric graph to play the role of a single cop on a planar graph. We also give an example of a geometric graph requiring 3 cops.

Recent years have witnessed significant interest in the study of random graph models, motivated by the need to understand complex real world networks. In this setting, the game of cops and robbers is a simplified model for network security. There are many recent results on cops and robbers on random graph models, including the Erd\H{o}s-Renyi model and random power law graphs \cite{bollobas+kun,luczak+pralat,pralat,bonato+pralat+wang,bonato+kemes+pralat, pralat+wormald}. We add to this list of stochastic models by considering cops and robbers on random geometric graphs. 
A \emph{random geometric graph} $G$ on $[0,1]^2$ contains of $n$ points drawn uniformly at random.
Two points $x,y \in V(G)$ are adjacent when the distance between them is within the connectivity radius, i.e.~$\norm{x - y} \leq r$. 
We denote the probability space of random geometric graphs by $\RGG$. 
Typically, we view the radius as a function $r(n)$, and then study the asymptotic properties of $\RGG$ as $n$ increases. 
We say that event $A$ occurs \emph{with high probability}, or \emph{whp}, when $\Pee[A] = 1 -o(1)$ as $n$ tends to infinity, or equivalently, $\lim_{n \rightarrow \infty} \Pee[A] = 1.$  
For example, $G \in \RGG$ is connected whp if $r = \sqrt{\frac{\log n+\kn}{n}}$. 
(Here and in the remainder of this paper, $\kn$ denotes an arbitrarily slowly growing function.) 
For this and further results on $\RGG$, see the monograph \cite{penrose}. 

We improve on the bound of Theorem \ref{thm:nine} when our random geometric graph is sufficiently dense.
Essentially, we determine thresholds for which we can successfully adapt known pursuit evasion strategies to the geometric graph setting. Typical analysis of $\RGG$ focusses on the homogeneous aspects of the resulting graph, resulting from tight concentration around the expected  structural properties. Our cop strategies rely on these homogeneous aspects.

When studying $G \in \RGG$, it is often productive to tile $[0,1]^2$ into small squares, chosen so that  whp, there is a vertex in each square, and vertices in neighboring squares are adjacent in $G$. We then use the induced grid  on these vertices to analyze properties of $G$, cf. \cite{avin+ercal,cooper+frieze}. It is easy to show that the 2-dimensional grid has cop number 2.  When our random geometric graph is dense enough,  we can adapt a winning two cop strategy on the grid to obtain a winning strategy on $\RGG$.

\begin{theorem}
\label{thm:two}
There is a constant $K_1>0$ such that the following holds.
If $G \in \RGG$ on $[0,1]^2$ with $r \geq K_1 (\log n/n)^{\frac14}$ then $c(G) \leq 2$ whp.
\end{theorem}

A further increase in the connectivity radius leads to an even denser geometric graph, so that eventually the cops and robbers game on $\RGG$ becomes quite similar to a turn-based pursuit evasion game on $[0,1]^2$. 
Such pursuit evasion games on $\R^d$ and in polygonal environments have been well studied, using winning criteria such as capture \cite{sgall, KR, bhadauria+isler} and line-of-sight visibility \cite{lavalle,guibas,isler05tro}.
 It is known  \cite{sgall, KR} that pursuers can win the capture game in $\R^d$ if and only if the evader starts in the interior of the convex hull of the initial pursuer locations. Furthermore, a single pursuer can always catch the quarry in a bounded region, such as $[0,1]^2$.  We use the dismantlable criterion of Theorem \ref{thm:pitfall} to prove that a sufficiently dense $\RGG$ also requires a single pursuer. 

\begin{theorem}
\label{thm:one}
There is a constant $K_2>0$ such that the following holds.
If $G \in \RGG$ on $[0,1]^2$ with  $r \geq K_2 \left(\log n/n\right)^{\frac15} $, then $c(G)=1$ whp.
\end{theorem}

\noindent
We note that Theorem \ref{thm:one} was proven independently by Alon and Pra{\l}at \cite{alon} using a graph pursuit algorithm in the spirit of  \cite{sgall, KR}.

Finally we also give a lower bound of the cop number of $\RGG$ proving that 
some random geometric graphs beyond the connectivity threshold require at least two cops.
This answers a question of Alon~\cite{alon}.
\begin{theorem}
\label{thm:multi}
There is a constant $K_3>0$ such that the following holds.
If $G \in \RGG$ on $[0,1]^2$ with  $r \leq K_3 \log n / \sqrt{n}$, then $c(G) > 1$ whp.
\end{theorem}
We do not know whether any of our multiple  cop bounds are tight. 
We are particularly hopeful that the bound for arbitrary geometric graphs can be improved.

\section{Notational conventions}
\label{sec:notation}

We begin by setting some notation.
For $x \in \R^2$ and $r \in \R$, define the ball $B(x,r) = \{ y \in \R^2 : \norm{x-y} \leq r \}$. 

In the standard formulation of cops and robbers, the cops are first to act in each round.  In continuous pursuit evasion games, the evader is usually first to act. 
The distinction is merely notational, and we choose to  view the robber as the first to act in each round. This leads to a more intuitive notation for the game state in our proofs below. Indeed, our cops are always reacting to the robber's previous move (which was made according to some unknown strategy), so it is useful to group these two moves together in a single round. 

We formally describe the game of cops and robbers using this notational convention. Before the game begins, the $\ell$ cops place themselves on the graph at vertices $C_1^0, \ldots , C_{\ell}^0$. Then the game begins. In the first round, the robber chooses his location $R^1$. Next the cops begin the chase, moving to vertices  $C_1^1, \ldots , C_{\ell}^1$ where $C_{j}^1 \in \cN(C_j^0)$. For $i \geq 2$, the $i$th round starts in configuration $(R^{i-1}, C_1^{i-1}, \ldots , C_{\ell}^{i-1})$. The robber is first to act, leading to configuration $(R^{i}, C_1^{i-1}, \ldots , C_{\ell}^{i-1})$ where $R^i \in \cN(R^{i-1})$ at the start of the $i$th cop turn. Next, the cops move simultaneously to yield configuration $(R^{i}, C_1^{i}, \ldots , C_{\ell}^{i})$ at the end of the $i$th round. The cops win if $C_k^i = R^i$ for some finite $i,k$. Otherwise the robber wins.

Finally, we note that the winning cop criteria has an equivalent formulation. Namely, the cops win if there are finite $i,k$ such that $R^i \in  \cN(C_k^{i-1})$. Indeed, $C_k$ would subsequently capture the evader on his $i$th move, achieving $C_k^i=R^i$.  Of course, if $R^i \notin  \cN(C_k^{i-1})$ for all $k$, then the robber cannot be caught in the current round, and his evasion continues.

\section{Geometric graphs}
\label{sec:nine}

In this section, we prove Theorem \ref{thm:nine}. 
Let $G=G(x_1, \ldots x_n;r)$ be a fixed geometric graph. 
We say that a cop $C$ \emph{controls} a path $P$ if whenever the robber steps onto $P$, then he steps onto $C$ or is caught by $C$ on his responding move. Let $\mathrm{diam}(G)$ denote the diameter of the graph. Aigner and Fromme \cite{aigner+fromme} prove the following.

\begin{lemma}[\cite{aigner+fromme}]
\label{lemma:control}
Let $G$ be any graph, $u,v \in V(G)$, $u \neq v$ and $P= \{u=v_0, v_1, \ldots v_s = v \}$ a shortest path between $u$ and $v$. A single cop $C$ can control $P$ after at most $\mathrm{diam}(G) + s$ moves.
\end{lemma}
It takes $C$ at most $\mathrm{diam}(G)$ moves to reach $P$, and then at most $s$ moves to take control of $P$.  We have the following simple corollary which will be useful for geometric graphs. 

\begin{cor}
\label{cor:3control}
Suppose that there are three cops $C_-, C, C_+$ chasing robber $R$ on $G$.
Consider a shortest $(u,v)$-path $P=\{u=v_0, v_1, \ldots,  v_s = v \}$. After $k \leq \mathrm{diam}(G) + 2s$ moves, the cop $C$  controls $P$, and $(C_-^k, C^k, C_+^k) = (v_{i-1}, v_i, v_{i+1})$,
where  we set $v_{-1} = u$ and $v_{s+1}=v$.
\end{cor}

\bpf
Start with the three cops colocated on any vertex of $P$. The cops attain this controlling configuration in two phases. In phase one, cops move as one until they control the path, as in Lemma \ref{lemma:control}. In phase two,  $C$ remains in control of the path while $C_-,C_+$ obtain their proper positions within $s$ moves. Assume that until round $j \geq 1$ of phase two, $C_+$ is colocated with $C$.  If $C$ stays put on $v_i$ in round $j$, then $C_+$ moves to $v_{i+1}$. If $C$ moves from $v_i$ to $v_{i-1}$ then $C_+$ stays put on $v_{i}$. Otherwise, both $C$ and $C_+$ move to $v_{i+1}$. After at most $s$ rounds, $C$ must either stay put or move left, and $C_{+}$ attains his proper position. Similarly, $C_-$ attains his position within $s$ rounds.
\epf

Geometric graphs are frequently non-planar. Because of  crossing edges, simply keeping $R$ from stepping onto $P$ does not necessarily prevent him from moving from one side of $P$ to the other. We say that $R$ \emph{crosses} $P$ at time $t$ if the closed segment $R^{t-1}R^t$ has nonempty intersection with the closed segments corresponding to the edges of $P$. The additional guards flanking $C$  ensure that once the three cops are positioned as in Corollary \ref{cor:3control},  $R$ cannot cross $P$. On a geometric graph, we say that a set of cops \emph{patrols} a path $P$ if they control $P$ and whenever $R$ crosses $P$, he is caught in the subsequent cop move. 

\begin{lemma}
\label{lemma:patrol}
Let $P=\{ v_0, \ldots , v_t \}$ be a shortest path on a geometric graph $G(x_1, \ldots , x_n; r)$. Suppose that the cops $C_-, C, C_+$ are located on $v_{i-1}, v_i, v_{i+1}$ respectively, and that cop $C$ controls $P$.  Then these three cops patrol $P$. 
\end{lemma}

\bpf
 If the robber steps onto $P$ then $C$ will capture him.  Suppose that the robber can cross $P$ without losing the game, and does so from position $R^{t}$ to $R^{t+1}$. 
 We characterize some constraints on the location of $R^{t}$. Consider the configuration $(R^{t}, C_-^{t-1}, C^{t-1}, C_+^{t-1})$ prior to robber's crossing. This  occurs in round $t$, after the robber move but before the cop moves. At this point, the cops are positioned on three successive vertices of $P$.
 We claim that  $R^{t} \notin B(C^{t},r)$. Indeed, if $C^{t-1} = C^{t}$ (so that the cops are stationary in round $t$), then $C$ can actually catch $R$ at time $t$, a contradiction. Otherwise  $C^{t} \in \{ C_-^{t-1}, C_+^{t-1} \}$, so one of these flanking cops can catch $R$ at time $t-1$, also a contradiction.
 
Next, we observe that  the robber cannot be far from the cops. 
Let $ (R^{t}, C_-^{t-1}, C^{t-1}, C_+^{t-1}) = (R^{t}, v_{i-1}, v_i, v_{i+1}).$
First of all,  $R^{t} \notin B(v_{i-2},r) \cup B(v_{i+2},r)$. Indeed, if $R^{t}$ is close to  either of $v_{i-2}, v_{i+2}$ then $R$ could step onto that vertex in round $t+1$  without being caught by $C$, contradicting the fact that $C$ controls $P$. Secondly, $R^t$ cannot be within $2r$ of any path vertex $v_j$ where $|i-j| >2$ by a similar argument. We conclude that the robber must cross $P$ between $v_{i-2}$ and $v_{i+2}$. 
The region forbidden to $R^t$ along this subpath is shown in Figure \ref{fig:nocross}(a).

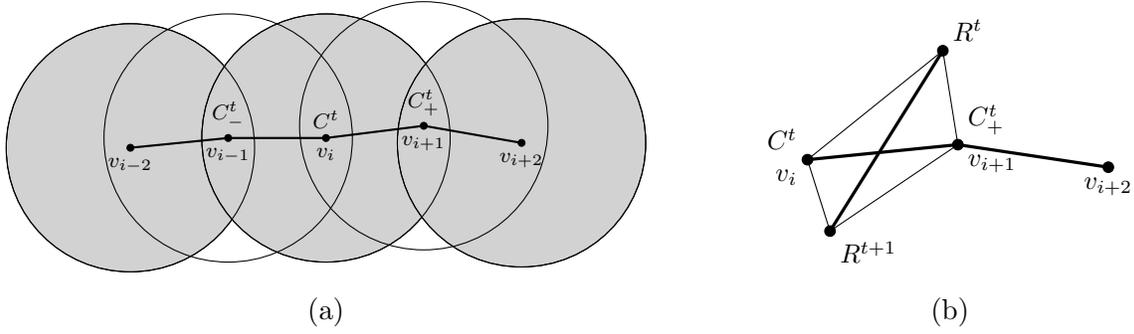
\begin{figure}[ht]
\begin{center}
\begin{tabular}{ccc}

\begin{tikzpicture}[scale=.65]
\tikzstyle{every node}=[font=\scriptsize]

\path(-4,-.2) coordinate (V0);
\path (-2,0) coordinate (V1) ;
\path (0,0) coordinate (V2);
\path (2, .25) coordinate (V3);
\path (4, -.1) coordinate (V4);

\draw[fill=gray!35] (V0) circle (1in);
\draw[fill=gray!35] (V2) circle (1in);
\draw[fill=gray!35] (V4) circle (1in);

\draw[thick]  (V0) -- (V1) -- (V2) -- (V3) -- (V4);

\foreach \j in {0,1,2,3,4}
{
\draw[fill] (V\j) circle (2pt);
\draw (V\j) circle (1in);
}

\node[above] at (V1) {${C_-^t}$};
\node[above] at (V2) {${C^t}$};
\node[above] at (V3) {${C_+^t}$};

\node[below] at (V0) {$v_{i-2}$};
\node[below] at (V1) {$v_{i-1}$};
\node[below] at (V2) {$v_{i}$};
\node[below] at (V3) {$v_{i+1}$};
\node[below] at (V4) {$v_{i+2}$};

\end{tikzpicture}

&
\hspace{.25in}
&

\begin{tikzpicture}[scale=1]
\tikzstyle{every node}=[font=\small]

\path (0,0.05) coordinate (V2);
\path (2, .25) coordinate (V3);
\path (4, -.05) coordinate (V4);

\path (1.8, 1.5) coordinate (R1);
\path (.3, -.9) coordinate (R2);

\draw[very thick]   (V2) -- (V3) -- (V4);

\foreach \j in {2,3,4}
{
\draw[fill] (V\j) circle (2pt);

}

\draw[fill] (R1) circle (2pt);
\draw[fill] (R2) circle (2pt);

\draw[very thick] (R1) -- (R2);

\draw (R1) -- (V2) -- (R2) -- (V3) -- cycle;

\node[above right] at (V3) {${C^{t}_+}$};
\node[above left] at (V2) {${C^{t}}$};

\node[below left] at (V2) {$v_{i}$};
\node[below right] at (V3) {$v_{i+1}$};
\node[below] at (V4) {$v_{i+2}$};

\node[above right] at (R1) {$R^{t}$};
\node[below right] at (R2) {$R^{t+1}$};

\end{tikzpicture}

\\
(a) && (b) \\

\end{tabular}
\caption{(a) The robber must cross between $v_{i-2}$ and $v_{i+2}$, but $R^{t}$  cannot lie in the  gray region
$B(v_{i-2},r) \cup B(v_i, r) \cup B(v_{i+2},r)$. (b) The geometry of the quadrilateral $v_i R^{t} v_{i+1} R^{t+1}$ shows that the robber cannot cross $P$ at edge $v_{i}v_{i+1}$ without ending in $B(C^{t},r) \cup B(C_+^{t},r)$.}
\label{fig:nocross}
\end{center}
\end{figure}

Without loss of generality, assume  that $R$ crosses $P$ so that $R^{t}R^{t+1}$ intersects $v_iv_{i+1}$ or $v_{i+1}v_{i+2}$.  Now  $R^{t+1} \notin  B(v_i, r) \cup B(v_{i+1}, r)$; otherwise either $C$ or $C_+$ can immediately catch him. 
Suppose that $R^{t}R^{t+1}$ crosses $v_{i}v_{i+1}$ where  $R^{t} \notin B(v_i, r) \cup B(v_{i+2},r)$ and $R^{t+1} \notin B(v_{i},r) \cup B(v_{i+1},r)$, as shown in  Figure \ref{fig:nocross}(b). 
We have $\norm{v_{i} - v_{i+1}} \leq r$ and $\norm{ R^{t} - v_{i}} > r$. This means that the angle $\measuredangle v_{i} R^{t} v_{i+1} < \pi/2$; otherwise in the triangle $v_{i} v_{i+1} R^{t}$, this obtuse angle forces  $r \geq \norm{v_{i} - v_{i+1}} > \norm{v_{i} - R^{t}} > r$, a contradiction.  Likewise, since $\norm{R^{t+1} - v_{i+1}} > r$, we must have  $ \measuredangle v_{i} R^{t+1} v_{i+1} < \pi/2$. Therefore $\max \{ \measuredangle R^{t} v_{i} R^{t+1}, \measuredangle R^{t} v_{i+1} R^{t+1} \} > \pi/2$, and the resulting obtuse triangle forces  $ \norm{R^{t} - R^{t+1}} > r$, a contradiction. Therefore $R$ cannot cross $P$ by crossing $v_{i}v_{i+1}$. An identical argument, replacing $v_i$ with $v_{i+2}$,  shows that $R$ cannot cross $v_{i+1}v_{i+2}$. Therefore, $R$ cannot cross $P$. 
\epf

We now prove that if $G$ is a connected geometric graph in $\R^2$, then $c(G) \leq 9.$

\bpfof{Theorem \ref{thm:nine}}
The proof is a direct adaptation of the Aigner and Fromme \cite{aigner+fromme} proof of Theorem \ref{thm:planar}. In our proof, we need 3 cops to patrol a shortest path of a geometric graph, instead of the single cop required to control a shortest path of a planar graph. 
The idea of the proof of Aigner and Fromme is divide the pursuit into stages. In stage $i$, we assign to $R$ a certain subgraph $H_i$, the \emph{robber territory}, which contains all vertices which $R$ may still safely enter, and to show that, after a finite number of cop-moves, $H_i$ is reduced to $H_{i+1} \subsetneq H_i$. Eventually, there is no safe vertex left for the robber. In each iteration, at most two shortest paths in $H_i$ must be controlled. For a planar graph, this requires one cop per path, and the third cop moves to control another shortest path in $H_i$. For geometric graphs, Lemma \ref{lemma:patrol} shows that 3 cops can patrol any shortest path of a geometric graph. Using that lemma  in place of  Lemma \ref{lemma:control}, the proof of Aigner and Fromme for planar graphs with 3 cops  becomes a proof for geometric graphs with 9 cops. See \cite{aigner+fromme} for the proof  details.
\epf

It is an open question whether this upper bound on the cop number can be improved for the class of geometric graphs. Here we construct a geometric graph that requires 3 cops, which leaves a considerable gap to our upper bound.
Aigner and Fromme \cite{aigner+fromme} proved that any graph with minimum degree $\delta(G) \geq 3$ and girth $g(G) \geq 5$ has $c(G) \geq \delta(G)$.
We describe a geometric graph $G$ on 1440 vertices with unit connectivity radius which has girth 5 and minimum degree 3, so that $c(G) \geq 3$.  A representative subgraph of $G$ appears in Figure \ref{fig:annular}.  Start with an annulus having inner radius 55 and outer radius 57. Within the annulus, we create an inner and outer strip of pentagons. Each pentagon corresponds to a one degree angle (or $\pi/180$ radians), so that there are a total of 720 pentagons. 
We give the vertex locations in polar coordinates $(r:\theta)$ where $\theta$ is in degrees.
For integral $\theta$, $1 \leq \theta \leq 360$, place a vertex at $(55:\theta)$ and at $(57:\theta + 1/2)$.  The interior points (separated by $1/2$ degree) are chosen in a clockwise repeating pattern  $(55: 2\theta)$, $(56.35: 2\theta + 0.5)$, $(55.85: 2\theta +1)$ and $(56: 2\theta + 1.5)$ for integral $\theta$, $1 \leq \theta \leq 180$.  Simple calculations show that a unit connectivity radius gives the geometric graph as shown in Figure \ref{fig:annular}. For example, the law of cosines calculates the lengths of edges on the outer and inner boundaries as approximately $0.995$ and $0.960$, respectively.


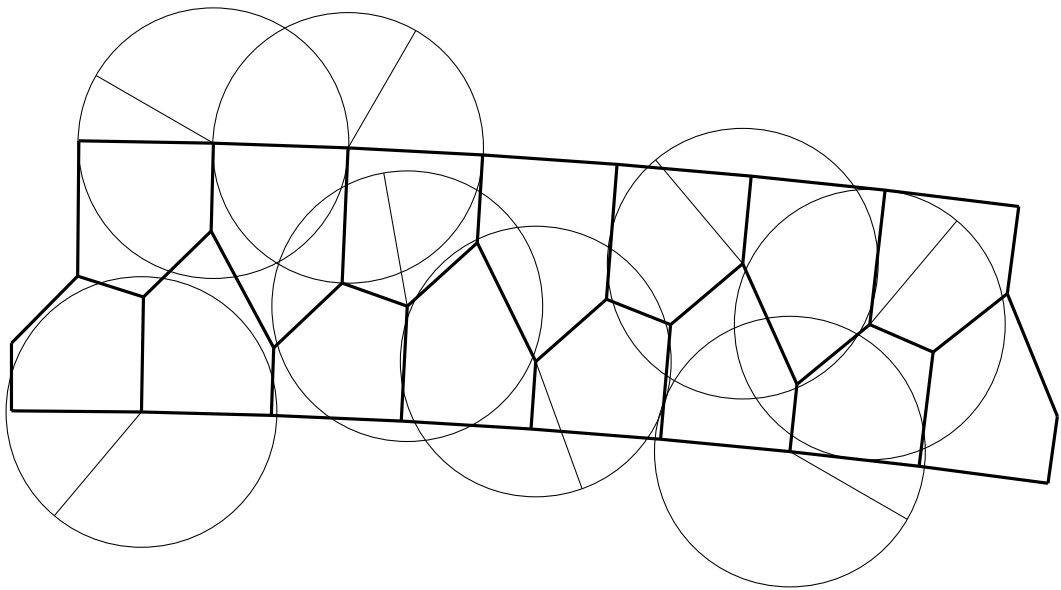
\begin{figure}[ht]
\begin{center}
\begin{tikzpicture}[scale=1.8, rotate=90, yscale=-1]

\tikzstyle{every node}=[font=\small]

\path (0: 55) coordinate (A1);
\path (1: 55) coordinate (A2);
\path (2: 55) coordinate (A3);
\path (3: 55) coordinate (A4);
\path (4: 55) coordinate (A5);
\path (5: 55) coordinate (A6);
\path (6: 55) coordinate (A7);
\path (7: 55) coordinate (A8);
\path (8: 55) coordinate (A9);

\path (.5: 57) coordinate (B1);
\path (1.5: 57) coordinate (B2);
\path (2.5: 57) coordinate (B3);
\path (3.5: 57) coordinate (B4);
\path (4.5: 57) coordinate (B5);
\path (5.5: 57) coordinate (B6);
\path (6.5: 57) coordinate (B7);
\path (7.5: 57) coordinate (B8);

\path (0: 55.5) coordinate (C1);
\path (1: 55.85) coordinate (C2);
\path (2: 55.5) coordinate (C3);
\path (3: 55.85) coordinate (C4);
\path (4: 55.5) coordinate (C5);
\path (5: 55.85) coordinate (C6);
\path (6: 55.5) coordinate (C7);
\path (7: 55.85) coordinate (C8);
\path (8: 55.5) coordinate (C9);

\path (.5: 56) coordinate (D1);
\path (1.5: 56.35) coordinate (D2);
\path (2.5: 56) coordinate (D3);
\path (3.5: 56.35) coordinate (D4);
\path (4.5: 56) coordinate (D5);
\path (5.5: 56.35) coordinate (D6);
\path (6.5: 56) coordinate (D7);
\path (7.5: 56.35) coordinate (D8);
\path (8.5: 56) coordinate (D9);

\draw[very thick] (A1) -- (A2) -- (A3) -- (A4) -- (A5) -- (A6) -- (A7) -- (A8) -- (A9);
\draw[very thick] (B1) -- (B2) -- (B3) -- (B4) -- (B5) -- (B6) -- (B7) -- (B8);

\foreach \i in {1,2,3,4,5,6,7,8}
{
\draw[very thick] (A\i) -- (C\i);
\draw[very thick]  (B\i) -- (D\i);
\draw[very thick]  (C\i) -- (D\i);
}

\draw[very thick]  (A9)--(C9);

\draw[very thick]  (D1) -- (C2);
\draw[very thick]  (D2) -- (C3);
\draw[very thick]  (D3) -- (C4);
\draw[very thick]  (D4) -- (C5);
\draw[very thick]  (D5) -- (C6);
\draw[very thick]  (D6) -- (C7);
\draw[very thick]  (D7) -- (C8);
\draw[very thick]  (D8) -- (C9);

\draw (A2) circle (1);
\draw(A7) circle (1);
\draw(B2) circle (1);
\draw(B3) circle (1);
\draw(C4) circle (1);
\draw(C5) circle (1);
\draw(D6) circle (1);
\draw(D7) circle (1);

\draw (A2) -- ++(220:1);
\draw (A7) -- ++(120:1);
\draw (B2) -- ++(300:1);
\draw (B3) -- ++(30:1);
\draw (C4) -- ++(350:1);
\draw (C5) -- ++(160:1);
\draw (D6) -- ++(320:1);
\draw (D7) -- ++(40:1);

\end{tikzpicture}
\caption{Part of a 3-regular geometric graph $G$ on 1440 vertices with $c(G)=3$. The eight circles show the connectivity neighborhood for each type of vertex.}
\label{fig:annular}
\end{center}
\end{figure}

We must have $c(G)=3$ since $G$ is planar. Indeed, there is a simple  winning strategy for three cops. Have cop $C_1$ remain stationary on any interior vertex. Place  cops $C_2,C_3$ on vertices on the inner and outer boundaries, separated by half a degree.
In each step, one of the boundary cops can take a clockwise step along his boundary while preventing the robber from crossing the shortest path between $C_2,C_3$. Eventually the robber cannot move counterclockwise because of $C_2,C_3$, and cannot move clockwise because of $C_1$.


\section{Adapting a grid strategy for $\RGG$}
\label{sec:four}
\setcounter{figure}{0}

In this section, we prove Theorem \ref{thm:two}. Our winning two cop strategy is similar to a winning strategy on the grid $P_n \Box P_m$. One cop catches the robber's ``shadow'' in a copy of $P_n$, while the other catches the robber's shadow in a copy of $P_m$. On subsequent moves, either the robber moves towards the boundary, or at least one cop decreases his distance from the robber. Eventually, the robber hits the boundary, and the cops close in for the win. Our cop strategy below follows along similar lines, but accommodates the full range of robber movement.

It is convenient to split the proof of Theorem~\ref{thm:two} into two parts, a probabilistic part and 
a deterministic part. Let $V= \{ x_1, \ldots , x_n \} \subset [0,1]^2$ and let $r \geq s > 0$. 
Let us say that the  tuple $(x_1,\dots,x_n;r,s)$ satisfies condition \conM~when  the following holds:

\begin{itemize}
 \item[\conM] For every $x \in [0,1]^2$ and every $y \in B(x,r) \cap [0,1]^2$, we have 
$V \cap B(x,r) \cap B(y,s) \neq \emptyset$.
\end{itemize}
 
\noindent
All the probability theory needed in the proof of Theorem~\ref{thm:two} is contained in the following lemma.

\begin{lemma}\label{lem:twoproba}
Let us set $s := 5 \sqrt{ \log n / n}$.
Let $x_1,\dots,x_n \in [0,1]^2$ be chosen i.i.d.~uniformly at random, and let $r\geq s$
be arbitrary. 
Then $(x_1,\dots,x_n;r,s)$ satisfies condition \conM~whp.
\end{lemma}

\begin{proof}
Let us set $t := 1 / \left\lceil\sqrt{ n / 2\log n}\right\rceil$.
Then $t = (1+o(1)) \sqrt{2 \log n / n}$ and it is of the form $t = 1/k$ with $k\in\eN$ an integer.
We can thus tile $[0,1]^2$ into $1/t^2$ squares of dimension $t\times t$.
Let $Z$ denote the number of these squares that do not contain any point of
$x_1,\dots,x_n$.
Then
\[ 
 \Ee [Z] 
= (1/t^2) \cdot (1 - t^2)^n 
\leq (1/t^2) e^{-n t^2}
= (1+o(1)) \frac{n}{2\log n} e^{-(1+o(1))2\log n}
= o(1).
\]
Thus, whp each square contains at least one $x_i$.

Now let us assume that each square of our dissection indeed contains
a point of $x_1,\dots,x_n$ and pick an arbitrary $x\in[0,1]^2$ and $y\in B(x;r)\cap [0,1]^2$.
If $\norm{x-y} < r-t\sqrt{2}$ then the square of our dissection that contains $y$ 
is completely contained in $B(x;r)$ (because the diameter of a $t\times t$ square is $t\sqrt{2}$). 
Hence any point $x_i$ that lies inside this square will clearly do as $\norm{y-x_i} \leq t\sqrt{2} < s$.
Let us thus assume $r-t\sqrt{2} \leq \norm{x-y} \leq r$, and let $z\in[x,y]$ 
be chosen on the segment between $x$ and $y$ in such a way that 
$\norm{z-x}= r-t\sqrt{2}$.
Then the square of our dissection that contains $z$ is contained
in $B(x;r)$ and the point $x_i$ inside this square satisfies
$\norm{y-x_i} \leq \norm{y-z}+\norm{z-x_i} \leq 2t\sqrt{2} \leq s$.
\end{proof}

\begin{lemma}\label{lem:twodet}
Suppose that $(x_1,\dots,x_n;r,s)$ with $x_1,\dots,x_n \in [0,1]^2$ and $0 < s < r^2/10^{10}$ satisfy condition \conM.
Then $c( G(x_1,\dots,x_n;r) ) \leq 2$.
\end{lemma}

\begin{proof}
We can assume without loss of generality
that $r \leq \sqrt{2}$ because otherwise $G$ is a clique and a single cop will be able to catch
the robber in a single move.
We start by describing the strategy of the cops.
The two cops act independently (i.e.~the action of $C_1$ does not
depend on the position or movement of $C_2$ and vice versa). First,  we
describe only the movements of $C_1$. Cop $C_2$ will follow a similar strategy, described below.

We introduce notation for a series of lines and points. Suppose the robber is at point $R^t$. Let $L_1^t$  be the vertical line through $R^t$. Let $P_1^t$ denote
the point on $L_1$ exactly $r/3$ below $R^t$ provided this point is above the $x$-axis.
Otherwise $P_1^t$ is the point on the $x$-axis exactly below $R_1^t$.
Similarly, we define the horizontal line $L_2^t$ and the point $P_2^t$ to the left of $R^t$ on $L_2$. For simplicity, we occasionally refer to $L_1, L_2, P_1, P_2$ (without the superscript) to refer to these lines and points with respect to the current position of $R$.

At time $t=0$, $C_1$ starts at a vertex $C_1^0 := x_j$ that is within $s$ of the origin $(0,0)^t$;  such an $x_j$ 
exists because of \conM.
In each round, the robber will first choose his new location $R^{t+1}$.
The cop then chooses a point $y \in B(C_1^t, r) \cap [0,1]^2$ and finds an 
$x_i \in B(C_1^t,r)\cap B(y,s)$ (such an $x_i$ exists because of property \conM)
and chooses as his new location $C_1^{t+1} := x_i$.
The strategy of $C_1$ has three phases:

\begin{my_itemize}
\item[S1:] Cop $C_1$ moves right
until he reaches a point within $s$ of $L_1$ and within $r/10^9$ of the $x$-axis.
\item[S2:] While staying within $r/10^{7}$ of $L_1$,
cop $C_1$ moves to within $s$ of the point $P_1$.
\item[S3:] Cop $C_1$ tries to stay as close to $P_1$ as he can.
\end{my_itemize}

{\bf Stage S1:} 
During stage S1, cop $C_1$ moves as follows.  
Let $y$ be the point of $B(C_1^t,r)$ closest to $L_1^{t+1}$. Then $C_1$ moves to a point $x_i \in B(C^t_1,r) \cap B(y,s)$.
If $y\in L_1$ then stage S1 ends. 
Otherwise, the cop travels a horizontal distance of at least $r-s$.
Thus, stage S1 lasts no more than $\lceil 1/(r-s)\rceil < 10/r$ rounds, since 
he can keep jumping right by at least $r-s$ and he will reach $L_1$ before he reaches the right boundary
of the unit square (note the cop either starts to the left of $L_1$ or within $s$ of $L_1$). 
Observe that, by the end of stage S1, the $y$-coordinate of $C_1$ is at most $s \cdot 10 / r < r/10^9$
(as $s < r^2 / 10^{10}$).

{\bf Stage S2: } In this stage, the cop will always stay as close to $L_1$ as he can, and will move
closer to his target point $P_1$ if he can.
The round starts with $C_1^t$ within $s$ of $L_1^t$ and within $r/10^9$ of the $x$-axis.
If $R^t$ has $y$-coordinate smaller than $r/3$ then we are immediately done with stage S2.
We can thus assume that $R^t$ is above $C_1^t$.

If $P_1^{t+1} \in B(C_1^t,r)$ then we can pick an $x_i \in B(C_1^t,r) \cap B(P_1^{t+1},s)$ and
set $C_1^{t+1} := x_i$, thereby ending stage S2.
Otherwise, the cop's move depends on how the robber moves.
We classify the possible robber moves into four (non-exclusive) types, depending on where the robber jumps, as shown in Figure \ref{fig:S2}. 
Writing this displacement in polar coordinates $(d:\theta)$, the four types  are

\begin{my_itemize}
\item[T1:] $d \leq r/2$. 
\item[T2:] $r/2 < d \leq r$ and $7\pi/6 \leq \theta \leq 11\pi/6$. 
\item[T3:] $r/2 < d \leq r$ and $2\pi/3 \leq \theta \leq 4\pi/3$. 
\item[T4:] $r/2 < d \leq r$ and $-\pi/6 \leq \theta \leq 2\pi/3$. 
\end{my_itemize}

\begin{figure}[h]
\begin{center}

\begin{tabular}{cccc}
\begin{tikzpicture}

\draw[thick,fill=gray!50] (0,0) circle (.5);

\draw (-1.25,0) -- (1.25,0);
\draw (0,-1.25) -- (0,1.25);

\draw[fill] (0,0) circle (.05);

\end{tikzpicture}
&

\begin{tikzpicture}
\draw[thick,fill=gray!50] (210:0.5) -- (210:1)
arc (210:330:1) -- (330:0.5)
arc (330:210:0.5) -- cycle;

\draw (-1.25,0) -- (1.25,0);
\draw (0,-1.25) -- (0,1.25);

\draw[fill] (0,0) circle (.05);

\end{tikzpicture}
&

\begin{tikzpicture}

\draw[thick,fill=gray!50] (120:0.5) -- (120:1)
arc (120:240:1) -- (240:0.5)
arc (240:120:0.5) -- cycle;

\draw (-1.25,0) -- (1.25,0);
\draw (0,-1.25) -- (0,1.25);

\draw[fill] (0,0) circle (.05);

\end{tikzpicture}
&

\begin{tikzpicture}

\draw[thick,fill=gray!50] (-30:0.5) -- (-30:1)
arc (-30:120:1) -- (120:0.5)
arc (120:-30:0.5) -- cycle;

\draw (-1.25,0) -- (1.25,0);
\draw (0,-1.25) -- (0,1.25);

\draw[fill] (0,0) circle (.05);

\end{tikzpicture}

\\
T1 & T2 & T3 & T4 
\end{tabular}
\end{center}
\caption{The robber move types.
In each case the robber will jump into the gray area.\label{fig:S2}}
\end{figure}
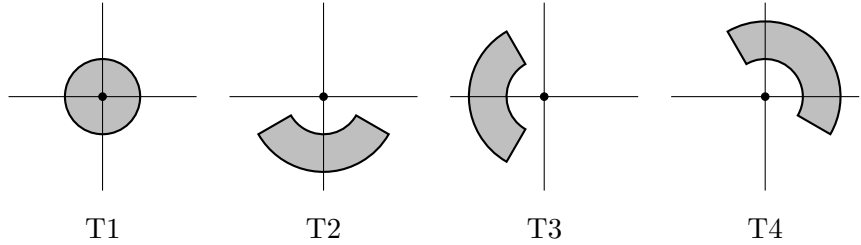

\noindent
If $R$ does a T1 move, then we compute $J^{t+1} := R^{t+1}-R^t$.
We can write $J^{t+1} = ( \ell \cos\alpha, \ell \sin\alpha)$ with $\ell  \leq r/2$.
Assuming $C_1^t$ is within $r/10^7$ of $L_1^t$, we can move at most
$r(\cos(\alpha)/2 + 1/10^{7})$ to the left or right to reach $L_1^{t+1}$.
Thus $y := \left(R_x^{t+1}, C_1^t + r \left(1-\frac12\cos\alpha-1/10^{7}\right) \right)^T \in L_1^{t+1}\cap B(C_1^t,r)$, where $R_x^{t+1}$ is the $x$-coordinate of $R^{t+1}$.
We pick $x_i \in B(C_1^{t+1},r)\cap B(y,s)$ and set $C_1^{t+1} := x_i$.
Observe $x_i$ is within $s$ of $L_1^{t+1}$ and that 
the distance between $C_1$ and $R$ has decreased by at least
$r\left(1-\frac12\sin\alpha-\frac12\cos\alpha-1/10^7\right) - s \geq 
r\left(1 - \frac12\sqrt{2}-1/10^7-1/10^{10}\right) > r / 4$.

If $R$ does a T2 move, then $L_1$ moves left or right by at most $r \cos(\pi/6) = \sqrt{3}r/2$ and
$R$ moves down by at least $r\sin(\pi/6) = r/2$.
Assuming that $C_1^t$ is within $r/10^7$ of $L_1^t$, we can thus move sideways by at most 
$(\sqrt{3}/2+1/10^7)r$ and reach $L_1^{t+1}$. We can therefore pick a point 
$y \in L_1^{t+1} \cap B(C_1^t,r)$ that is at least
$(\frac32-\sqrt{3}/2-1/10^7) r - s > r/2$ closer to $R^{t+1}$ than $C_1^t$ is to $R^t$.
Again we pick $x_i \in B(C_1^{t+1},r)\cap B(y,s)$ and set $C_1^{t+1} := x_i$.

If $R$ does a T3 or T4 move then we compute $y := R^{t+1}-R^t + C_1^t$,
(if $y \not\in [0,1]^2$ then we take the point $y'$ on $\partial[0,1]^2$ with minimum
distance to $y$)
we pick $x_i \in B(C_1^t,r)\cap B(y,s)$ and we set $C_1^{t+1} := x_i$.
Note that this way the distance of $C_1$ to $P_1$ cannot increase by more than
$s$.

{\bf Stage S3: } At present it is not yet clear whether stage S2 will ever finish (and also
we may not be able to stay within $r/10^7$ of $L_1$ indefinitely).
If however we do get to stage S3, 
we observe that  $R$ cannot make
a T1 or T2 move without getting caught by the cop immediately
(see Figure~\ref{fig:S2onedown}).
Therefore, during stage S3,  we act exactly as in the case of stage S2 where
$R$ does a T3 or T4 move.
This concludes the description of the first cop's movements.

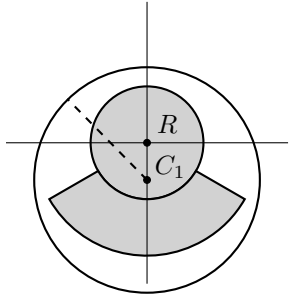
\begin{figure}[h]
 \begin{center}

 \begin{tikzpicture}[scale=1.5]

\draw[thick] (0,-.33) circle (1);
\draw[thick,fill=gray!35] (0,0) circle (.5);
\draw[thick,fill=gray!35] (210:0.5) -- (210:1)
arc (210:330:1) -- (330:0.5)
arc (330:210:0.5) -- cycle;

\draw (-1.25,0) -- (1.25,0);
\draw (0,-1.25) -- (0,1.25);

\draw[dashed,thick] (0,-.33) -- (-.707,.39);

\node[above right] (R) at (0,0) {\small{$R$}};
\node[below right] (C) at (-.02,-.03) {\small{$C_1$}};

\draw[fill]  (0,0) circle (.03);
\draw[fill]  (0,-.33) circle (.03);

\end{tikzpicture}

 \end{center}
\caption{If $C_1$ is within $r/10^7$ of the point $r/3$ below $R$, then 
$R$ can no longer make T1 or T2 moves.\label{fig:S2onedown}}
\end{figure}

Suppose that during the first $T = 1000 / r$ moves of the game the robber does not get caught.
Stage S1 will have finished after at most $10/r$ moves.
Since $s \cdot T < r/10^7$, we will be able to stay within $r/10^7$ of $L_1$ for
the remaining moves until $T$, and assuming we reach stage S3 at some time $t < T$ we will be able to stay within $r/10^7$
of $P_1$ for the remaining moves until $T$.
Thus stage S2 will have finished as soon as we have done at most $14/r$ moves of type T1 or T2
(the first $10/r$ may occur during stage S1 and after that we move closer to $P_1$ by at least
$r/4$ in each T1 or T2 move).  Thus, out of the first $T$ moves, at most $14/r$ robber moves are of type T1 or T2.

Completely analogously we can define a strategy for the second cop $C_2$ that will ensure that
in the first $T$ moves no more than $14/r$ robber moves are of type T1 or T3.
Cop $C_2$ tries to  attain position on the horizontal line $L_2$ through $R$. The stages of his strategy are:
\begin{my_itemize}
\item[S$'$1:] Cop $C_2$ moves up
until he reaches a point within $s$ of $L_2$ and within $r/10^9$ of the $y$-axis.
\item[S$'$2:] While staying within $r/10^{7}$ of $L_2$,
cop $C_2$ moves to within $s$ of the point $P_2$.
\item[S$'$3:] Cop $C_2$ tries to stay as close to $P_2$ as he can.
\end{my_itemize}

Observe that whenever $R$ does a T4 move, then the sum of his coordinates
increases by at least 
\[  
\min_{-\pi/6 \leq \theta \leq \frac{2\pi}{3}} (\sin\theta+\cos\theta) \frac{r}{2} 
 =  \left(\frac{\sqrt{3}-1}{4}\right) r.
\]
Meanwhile, if the robber makes a $T1$, $T2$ or $T3$ move, the sum of his coordinates decreases by at most $r \sqrt{2}$ 
(achieved at $\theta=5 \pi /4$).
Hence, if the robber did not get caught in the first $T$ moves, then the sum of robbers coordinates at time $T$ is at least
\[ 
\begin{array}{rcl}
 R_x^T + R_y^T 
& \geq & 
(T-28/r) \cdot \left(\frac{\sqrt{3}-1}{4}\right) r - (28/r) \cdot r\sqrt{2}  \\
& = & 
972 \left(\frac{\sqrt{3}-1}{4}\right)  - 28 \sqrt{2} 
 \,  > \,  
2.
\end{array} 
\]
But this is impossible, since the robber stays inside the unit square.
It follows that $R$ gets caught by the cops within the first $T$ moves.
\end{proof}

\begin{proofof}{Theorem \ref{thm:two}} Follows from Lemmas~\ref{lem:twoproba} and~\ref{lem:twodet} by taking $K_1=3 \cdot 10^5$.
\end{proofof}


\section{A dismantlable $\RGG$}
\setcounter{figure}{0}

In this section, we prove Theorem \ref{thm:one} by showing that when $r \geq K_2 ( \log n /n)^{1/5}$ the  random geometric graph is dismantlable.
We begin by setting some notation. Let $c := (\frac12,\frac12)$ denote the center of the unit square $[0,1]^2$.
Let us write 
\[ 
\nb(i) := \{ 1 \leq j \leq n : \norm{x_i-x_j}\leq r, \text{ and } \norm{x_j-c} < \norm{x_i-c} \}.
\]
In other words, $\nb(i)$ is the set of (indices) of vertices 
adjacent to $x_i$ and closer to the center $c$ than $x_i$.
We will prove the following lemma.

\begin{lemma}
\label{lemma:pitfall}
There is a constant $K_2>0$ such that the following holds.
Suppose $r \geq K_2 \left(\log n / n\right)^{1/5}$.
Whp~the following holds for all $1\leq i \leq n$:
either $\norm{x_i-c} < r/2$, or
there is a $j \in \nb(i)$ such that $\nb(i)\subseteq\nb(j)$.
\end{lemma}
Assuming that Lemma \ref{lemma:pitfall} holds, the proof of Theorem \ref{thm:one} is straightforward dismantling of the random geometric graph.

\begin{proofof}{Theorem~\ref{thm:one}}
We can induce a strict ordering of the vertices according to their distance from the center $c$, in descending order. Indeed, for any vertices $x,y$, $\Pee( \norm{x-c} = \norm{y-c})=0$.
By Lemma \ref{lemma:pitfall}, the outermost vertex is a pitfall, and can be removed. We continue to remove vertices until the remaining vertices lie in $B(c, r/2)$. The graph induced by these remaining vertices forms a clique, which is dismantlable. By Theorem \ref{thm:pitfall}, the graph has $c(G)=1$.
\end{proofof}

The remainder of this section is devoted to proving Lemma \ref{lemma:pitfall}, which requires a series of intermediate geometric lemmas. 
For $x,y\in\eR^2$, let us write
\begin{equation}
\label{eqn:W}
W(x,y;r) := \{ z \in \eR^2 : B(z,r) 
\supseteq B(x,r)\cap B(y,\norm{x-y}) \}.
\end{equation}
Let $[x,y]$ denote the line segment between these two points. Note that
\begin{equation}\label{eq:Anonincr}
\text{if } z \in [x,y] \text{ then }
W(x,y;r) \subseteq W(x,z;r).
\end{equation}
Indeed, we have
$B(x,r) \cap B(z, \norm{x-z} ) \subseteq B(x,r)\cap B(y,\norm{x-y})$ so that $W(x,y;r) \supseteq 
W(x,z;r)$. 
Observe that $\area(W(x,y;r))$ does not depend on the exact locations of $x,y$, but only on $\norm{x-y}$ and $r$. We can thus denote $A(d,r) := \area(W(x,y;r))$ for an arbitrary pair $x,y$ with $\norm{x-y}=d$.
By observation~\eqref{eq:Anonincr}, the area $A(d,r)$ is nonincreasing in $d$ for a fixed $r$.

We give a simpler geometric characterization of $W(x,y;r)$ when $\norm{x-y} = d > r$.
Let $p_1, p_2$ denote the two intersection points
of $\partial B(x,r)$ and $\partial B(y,d)$. Denote 
\[ 
W'(x,y;r) := B(p_1,r)\cap B(p_2,r),
\]
as shown in Figure \ref{fig:W'}(a).

\begin{figure}[h!]
\begin{center}
\begin{tabular}{ccc}
\begin{tabular}{c}
\begin{tikzpicture}

\path (104:1.35) coordinate (P1);
\path (-104:1.35) coordinate (P2);
\path (-2.5,0) coordinate (Y);
\path (0,0) coordinate (Z);

\draw (Y) circle (2.5);
\draw (Z) circle (1.35);
\draw (P1) circle (1.35);
\draw (P2) circle (1.35);

\begin{scope}[shift  = (P2)]
	\draw[thin, gray,fill=gray!50] (78:1.34) arc (78:102:1.34) -- cycle;
	\draw[thick]  (78:1.34) arc (78:102:1.34);
\end{scope}

\begin{scope}[shift  = (P1)]
	\draw[thin,gray,fill=gray!50] (-78:1.34) arc (-78:-102:1.34) -- cycle;
	\draw[thick]  (-78:1.34) arc (-78:-102:1.34);
\end{scope}

\draw[fill] (Y) circle (2pt);
\draw[fill] (Z) circle (2pt);
\draw[fill] (P1) circle (2pt);
\draw[fill] (P2) circle (2pt);

\node[left] at (Y) {$y$};
\node[right=7pt] at (Z) {$x$};
\node[above=5pt] at (P1) {$p_1$};
\node[below=5pt] at (P2) {$p_2$};
\node at (-.35,.26) {$W'$};

\end{tikzpicture}

\end{tabular}

& \hspace{.5in} &
\begin{tabular}{c}

\begin{tikzpicture}

\path (1.5,0) coordinate (Y);
\path (0,0) coordinate (Z);

\draw[fill=gray!50] (Y) circle (1);
\draw (Z) circle (2);

\draw[very thick] (-30:2) arc (-30:30:2);

\draw[fill] (-30:2) circle (1pt);
\draw[fill] (30:2) circle (1pt);

\node at (0, .2) {$D'$};
\node at (1.5, 0) {$D$};

\end{tikzpicture}
   \end{tabular}
   \\
  (a) & & (b)
\end{tabular}
\end{center}
\caption{(a) The set $W'=W'(x,y;r)$. (b) When closed discs intersect,  the smaller disc $D$
contains the shortest arc on the bigger disc $D'$ between the 
intersection points of the boundaries.}
\label{fig:W'}
\label{fig:DDprime}
\end{figure}
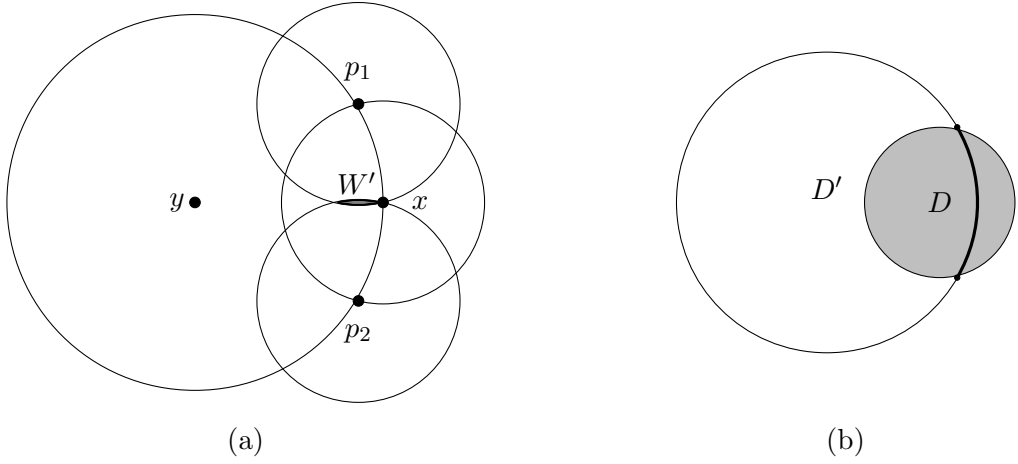

\begin{lemma}
If $\norm{x-y} = d > r$ then $W'(x,y;r) = W(x,y;r).$
\end{lemma}

\begin{proof}
Pick any $z\in W(x,y;r)$. We must have $p_1, p_2 \in B(z,r)$, which means that $z \in B(p_1,r) \cap B(p_2,r)$ Therefore $W(x,y;r) \subseteq W'(x,y;r)$.

Picking any $z \in W'(x,y;r)$, we have $p_1,p_2 \in B(z,r)$. 
Observe that if a closed disc $D$ intersects a disc $D'$ of the same or larger radius
then $D$ contains the shortest circular arc along $\partial D'$ between the two intersection points
of $\partial D$ and $\partial D'$, see Figure~\ref{fig:DDprime}(b).
So $B(z,r)$ contains the part of $\partial B(x,r)$ between $p_1$ and $p_2$ that lies inside
$B(y,d)$. Using that $d > r$, $B(z,r)$ also contains the part of $\partial B(y,d)$ between $p_1$ and $p_2$ that falls inside $B(x,r)$. 
Thus $B(z,r)$ contains $\partial\left( B(x,r)\cap B(y,\norm{x-y}) \right)$.
Because both $B(z,r)$ and $B(x,r)\cap B(y,\norm{x-y})$ are convex, it now also 
follows that $B(x,r)\cap B(y,d) \subseteq B(z,r)$. This shows that $W'(x,y;r) \subseteq W(x,y;r)$.
\end{proof}

We now compute a lower bound for $A(d,r)$ for distant vertices $x,y$.

\begin{lemma}
\label{lemma:AprimeLB} 
If $d=K \cdot \max \left( r, 1/\sqrt{2}  \right)$ where $K > 1$ is a sufficiently large constant,  then $A(d,r) =  \Omega(r^5).$ 
\end{lemma}

\begin{proof}
Choose $x,y \in \R^2$ with $\norm{x-y} = d$. 
The geometry of $W=W(x,y,r)$ is shown in Figure~\ref{fig:dralpha}.
We have

\begin{equation}\label{eq:Aprimesin}
\begin{array}{rcl}
\area(W) & = & 
4 \left( \pi r^2 \left( \displaystyle{\frac{\alpha}{2\pi}} \right) 
-   \frac12 r^2 \cos(\alpha)\sin(\alpha) \right)\\
&  = & 
r^2 \cdot \left( 2\alpha - \sin(2\alpha) \right).
\end{array}
\end{equation}
Indeed, the expression $\pi r^2 \left( \frac{\alpha}{2\pi} \right)$ equals the
area of a slice of opening angle $\alpha$ out of a disc of radius $r$, and
the term  $\frac12 r^2 \cos(\alpha)\sin(\alpha)$ equals
the area of a triangle with sides
$h = r \cos(\alpha)$ and $s = r \sin(\alpha)$.
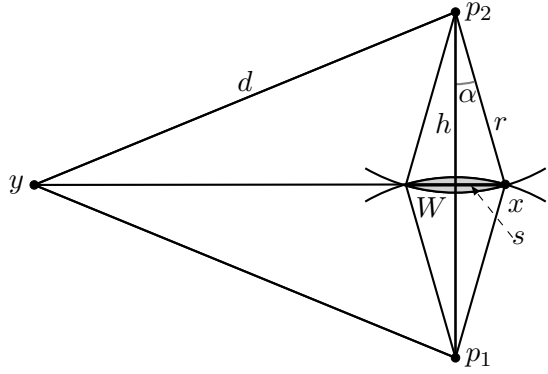
\begin{figure}
\begin{center}
\begin{tikzpicture}[scale=.8]

\path (0,0) coordinate (P1);
\path (0,5.75) coordinate (P2);
\path (-7,2.875) coordinate (Y);
\path (74:3) coordinate (Z);

\draw[thick] (60:3) arc (60:120:3);
\draw[thick] (106:3) -- (0,0) -- (74:3);
\draw[thick,fill=gray!35] (74:3) arc (74:106:3) -- cycle;

\begin{scope}[shift  = (P2)]
	\draw[thick] (-60:3) arc (-60:-120:3);
	\draw[thick] (-106:3) -- (0,0) -- (-74:3);
	\draw[thick,fill=gray!35] (-74:3) arc (-74:-106:3) -- cycle;
	\draw[thick,gray] (-90:1.2) arc (-90:-75:1.2);
	
	\node at (-81.5:1.4) {$\alpha$};
	
\end{scope}   

\draw[thick] (P1) -- (P2) -- (Y) -- cycle;
\draw[thick] (P1) -- (P2) -- (Y) -- cycle;
\draw[thick] (Y) -- (Z);

\draw[fill] (Y) circle (2pt);
\draw[fill] (Z) circle (2pt);
\draw[fill] (P1) circle (2pt);
\draw[fill] (P2) circle (2pt);

\node[left] at (Y) {$y$};
\node[below right] at (.7,2.8) {$x$};
\node[right] at (P1) {$p_1$};
\node[right] at (P2) {$p_2$};

\node at (-3.5,4.6) {$d$};
\node at (-.2,3.9) {$h$};
\node at (.75,3.9) {$r$};

\node at (-.4,2.5) {$W$};
\node at (1.05,2.0) {$s$};
\draw[dashed,-latex] (.95,2.0) -- (.25, 2.875);

\end{tikzpicture}
\end{center}
\caption{Determining the area of $W=W(x,y;r)$.\label{fig:dralpha}}
\end{figure}
Also note that
$d^2  =  h^2 + (d-s)^2$ and $r^2  =  h^2 + s^2$,
giving
\[ 
s = r^2 / 2d  = \min \left( \frac{r^2}{ K\sqrt{2}}, \frac{r}{2K} \right) = \Omega( r^2 ). 
\]
Thus,  $\sin(\alpha) = s / r = \Omega( r )$, and 
because $\sin(x) = x + o(x^3)$, this also gives $\alpha = \Omega(r)$.
The approximation  $x - \sin(x) = x^3 / 6 + o(x^5)$, together with~\eqref{eq:Aprimesin}, 
proves the lemma.
\end{proof}

Our next lemma places a lower bound on $\area(W(x,c;r))$ where $c=(\frac{1}{2}, \frac{1}{2})$ is the center of the unit square.
\begin{lemma} 
\label{lemma:Wprimequart}
For all $x \in [0,1]^2$ with $\norm{x-c} \geq r/2$, we have
$$
 \area\left( W(x,c;r) \cap [0,1]^2 
\cap B(c, \norm{x-c}) \right) = \Omega(r^5). $$

\end{lemma}

\begin{proof}
Pick the point $\ctil$ on the line $L$ containing $x$ and $c$, so that $c \in [\ctil, x]$ and
$\norm{x-\ctil} = d = K \cdot \max( r, 1/\sqrt{2} )$, see  Figure~\ref{fig:ctil}.
\begin{figure}[h]
 \begin{center}
\begin{tikzpicture}[scale=.7]

\path (0,0) coordinate (X);
\path (-6,0) coordinate (Y);
\path (-1.25,-1) coordinate (C1);
\path (218:6) coordinate (C2);
\path (38:1.75) coordinate (L);

\draw[fill] (X) circle (2pt);
\draw[fill] (C1) circle (2pt);
\draw[fill] (C2) circle (2pt);

\draw[dashed] (X) ++ (160:6) arc (160:225:6);
\draw[dashed] (X) -- (Y);

\draw (-3.25,-3) -- (.75,-3) -- (.75,1) -- (-3.25,1) -- cycle;

\draw (218:7) -- (38:2);

\node[below right] at (X) {$x$};
\node[below right] at (C1) {$c$};
\node[above=2pt] at (C2) {$\ctil$};
\node[below right] at (L) {$L$};

\node[above] at (-2.5,0) {$d$};

\node[below] at (-3.25,-3) {$(0,0)$};
\node[above] at (.75,1) {$(1,1)$};

\end{tikzpicture}
 \end{center}
\caption{Choosing $\ctil$ such that $c\in[\ctil,x]$.\label{fig:ctil}}
\end{figure}
By equation~\eqref{eq:Anonincr},   $W(x,\ctil;r) \subseteq W(x,c;r)$.
Provided that $K$ is sufficiently large,  
we have $\diam(W(x, \ctil; r)) < r / 10^{10}$. Furthermore, both the angle between $\partial B(p_1,r)$ and
the line $L$ at their intersection points, and the angle between $\partial B(p_2,r)$ and the line $L$ 
at their intersection points will be less than 1 degree.
It follows directly that
$ W(x,\ctil;r) \subseteq [0,1]^2 \cap B(c,\norm{x-c})  $
for every $x  \in [0,1]^2 \setminus B(c,r/2)$.
Applying Lemma \ref{lemma:AprimeLB} completes the proof.
\end{proof}

We conclude this section with the proof of our main lemma: that
for every vertex $x_i$ such that $\norm{x_i - c} > r/2$, there is a $j \in \nb(i)$ such that $\nb(i)\subseteq\nb(j)$.

\begin{proofof}{Lemma \ref{lemma:pitfall}}
We can assume without loss of generality that
$r \leq \sqrt{2}$ (otherwise $\norm{x_i-c} < r/2$ holds trivially for all $i$).
Let $Z$ denote the number of indices $i$ such that $\norm{x_i-c} \geq r/2$ and there
is no $j\in\nb(i)$ such that $\nb(j) \supseteq \nb(i)$.
Then $\Ee Z$ can be bounded above by:
\[
\begin{array}{rcl}
\Ee [Z] 
& \leq & 
\displaystyle 
n \int_{[0,1]^2 \setminus B(c,r/2)}
\left( 1 - \area( W(x,c;r) \cap [0,1]^2 ) \right)^{n-1} \dd \! x \\
& \leq & 
\displaystyle 
n \left(1-\Omega( r^5 )\right)^{n-1} 
\, \leq  \, 
\displaystyle
n \exp\left[ - \Omega( n r^5 ) \right] 
\end{array} 
\]
Thus, if we chose $K_2$ sufficiently large we have $\Ee Z \leq  \exp[ \log n - \Omega( n r^5 ) ]
= \exp[ - \Omega(\log n) ] = o(1)$.
So the assertion of the lemma holds whp.
\end{proofof}


\section{$\RGG$ near the connectivity threshold is not cop-win}
\setcounter{figure}{0}

In this section, we prove that some random geometric graphs require at least two cops. In particular,  when  we are near the connectivity threshold,  the graph is not dismantlable whp.

\begin{proofof}{Theorem \ref{thm:multi}}
Without loss of generality we can assume  $r \geq \frac12\sqrt{\log n/n}$, because by a result of Penrose~\cite{penroseMST}
our graph is disconnected whp for smaller choices of $r$ (obviously a disconnected graph is not cop-win). 
We will show that there is a small constant $K_3>0$ such that if $r \leq K_3 \log n / \sqrt{n}$  then
whp the graph is not dismantlable. 

Intuitively, we are hunting for a subset of $[0,1]^2$ as shown in Figure \ref{fig:necklace}. 
Start with an $N$-gon with side length $\rho_1$, slightly smaller than $r$. 
Draw a small disc $B(c_i, \rho_2)$ around each corner, where $\rho_1+2\rho_2 = r.$ 
We want each  disc $B(c_i, \rho_2)$ to contain exactly one vertex of $G$, say $x_i$. 
Next, we consider the sets $B(x_{i-1}, r) \cap B(x_{i+1},r)$. We want this intersection to  
contain no other vertices besides $x_i$. If we can find such a structure, it creates a cycle $\{x_1, \ldots x_N\}$ in $G$ 
such that $x_i$ the only vertex in $G$ that is adjacent to both $x_{i-1}, x_{i+1}$ (addition modulo $N$). 
Therefore $G$ is not dismantlable because none of the $x_i$ will ever become pitfalls.

\begin{figure}[h]
\begin{center}

\begin{tikzpicture}[scale=1.5]
       
\foreach \x in  {-45,0, 45,90} {
	\draw[fill=gray!35] (\x:2) circle (.2);
}       

\draw[color=black!75]   (-90:2) \foreach \x in { -45,0,45,90,135} {
                -- (\x:2)
           };

\path(-38:2.5) coordinate (C0);
\path(-10:2.8) coordinate (C1);
\path(-12:2.8) coordinate (CC1);
\path(55:2.2) coordinate (C2);
\path(90:.5) coordinate (W);

\path(-43.5:2.07) coordinate (X0);
\path(-1:1.87) coordinate (X1);
\path(41.5:2.07) coordinate (X2);

\node[below=3pt] at (C0) {$B(c_{i-1}, \rho_2)$};
\node at (CC1) {$B(c_i, \rho_2)$};
\node[right] at (C2) {$B(c_{i+1}, \rho_2)$};
\node[above] at (W) {$B(x_{i-1} ,r) \cap B(x_{i+1}, r)$};

\draw[-latex] (W) -- (13:.7);
\draw[-latex] (C1) -- (0:2.2);

\draw[fill] (X0) circle (1pt);
\draw[fill] (X1) circle (1pt);
\draw[fill] (X2) circle (1pt);

\node[right=2.5pt] at (X0) {$x_{i-1}$};
\node[left=2.5pt] at (X1) {$x_i$};
\node[right=2.5pt] at (X2) {$x_{i+1}$};

\node at (70:2.05) {$\rho_1$};

\begin{scope}[shift  = (X0)]
	\draw[thick] (45:1.75) arc (45:145:1.75);
	\draw[dashed] (0,0) -- (135:1.75);
	\node at (146:.8) {$r$};
\end{scope}   
	
\begin{scope}[shift  = (X2)]
	\draw[thick] (-30:1.75) arc (-30:-135:1.75);
	\draw[dashed] (0,0) -- (-40:1.75);
	\node at (-21:1) {$r$};
\end{scope}   

\end{tikzpicture}

\caption{For an $N$-gon with side length $\rho_1$, we want each $B(c_i,\rho_2)$ to contain a single vertex, and we want each  $B(x_{i-1} ,r) \cap B(x_{i+1}, r)$ to contain no additional vertices.}
\label{fig:necklace}
\end{center}

\end{figure}
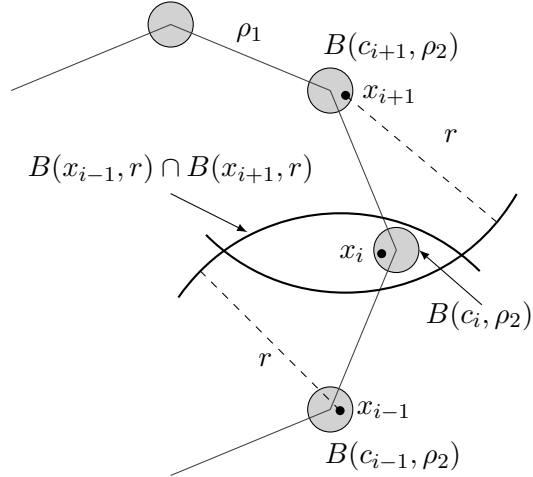
We now prove the existence of such a structure.  Let $N$ denote the number of vertices of the cycle; we will specify this value later. 
Set $\rho_1 = r - r/N^2$ and $\rho_2 = r/2N^2$.
Consider a regular $N$-gon $\Gamma \subseteq [0,1]^2$, whose edges each have 
length $\rho_1$. 
(Once we fix our choice of $N$, we shall see later that $\Gamma$ fits easily inside the unit square $[0,1]^2$.)
Let us label the corners of $\Gamma$ as $c_0, \dots, c_{N-1}$
for convenience, where of course $c_i$ is next to $c_{i-1}$ and $c_{i+1}$ (addition 
of indices modulo $N$).
We will insist that, for each $0\leq i \leq N-1$ there is a
point $x_{j_i} \in B(c_i, \rho_2)$ with
\begin{equation}
\label{eqn:in-small-ball}
\{x_1,\dots,x_n \} \cap B(c_{i},\rho_2) = \{x_{j_i}\}, 
\end{equation}
and the point $x_{j_i}$ is also the unique 
common neighbor of the two points $x_{j_{i-1}}$ and $x_{j_{i+1}}$, i.e.
\begin{equation}
\label{eqn:in-intersect}
\{ x_1, \dots, x_n \} \cap B(x_{j_{i-1}},r) \cap B(x_{j_{i+1}},r)
= \{ x_{j_i} \}.
\end{equation}
Observe that
\begin{eqnarray*}
\norm{c_{i+1}-c_{i-1}} 
& = & 
2 \rho_1 \sin \left(\frac{\pi(N-2)}{2N} \right) \,= \, 2 \rho_1 \cos \left( \frac{\pi}{N}\right)
\\
& =  &
2r\left(1-{1/N^2} \right)\left(1 - O\left(1/N^2\right)\right) 
\, = \, 
2r - O \left({r/N^2} \right)
\end{eqnarray*}
using the Taylor approximation $\cos(x) = 1 - \frac12 x^2 + O(x^4)$. Hence
for any $x\in B(c_{i+1},\rho_2)$ and $y \in B(c_{i-1},\rho_2)$ we also have
$\norm{x-y} = 2r - O(r/N^2)$.
Let us write $W(x,y) := B(x,r) \cap B(y,r)$.
By the same computation as equation \eqref{eq:Aprimesin},
\[
\area(W(x,y)) = r^2 (2\beta - \sin(2\beta)) = O( r^2 \beta^3 ), 
\] 
where $\beta$ is a small angle with $\cos\beta = \frac12\norm{x-y} / r = 1 - O(1/N^2)$, so that 
$\beta = O(1/N)$ (again using the Taylor expansion of cosine), Hence
\begin{equation}
\label{eqn:area-lower-bound}
\area(W(x,y)) = O( r^2 / N^3 ). 
\end{equation}

Rather than computing directly in the standard random geometric graph, it helps
to consider a ``Poissonized'' version.
Consider an infinite sequence $x_1,x_2, \dots$ of random points, i.i.d.~uniformly at random
on the unit square.
The ordinary random geometric graph, which we will denote by $G_O$ for the rest of the
proof, is just $G(x_1,\dots,x_n;r)$.
Now let $Z \isd \Po(n)$ be a Poisson random variable of mean $n$, independent of
the points $x_1,x_2,\dots$ and consider the random geometric graph $G(x_1,\dots, x_Z;r)$ on the
points $x_1,\dots,x_Z$ which we will denote by $G_P$.
Observe that the points $x_1,\dots,x_Z$ constitute a Poisson process of intensity 
$n$ on the unit square, which has the convenient properties that 
for every $A \subseteq [0,1]^2$ the number of points
that fall in $A$ is a Poisson random variable with mean $n\cdot\area(A)$, and
that for any two disjoint sets $A,B$ the number of points in $A$ is independent
of the number of points in $B$ (cf.~\cite{kingmanboek}).
This makes $G_P$ slightly easier to handle than $G_O$.
We shall first do our probabilistic computations for the Poissonized version $G_P$ and then we'll derive
the results for the original model $G_O$ from those for the Poissonized one.

Let us say the polygon $\Gamma$ is \emph{good} if it satisfies the demands
of equations \eqref{eqn:in-small-ball} and \eqref{eqn:in-intersect}  with $Z$ swapped for $n$.
Employing the useful independence properties of the Poisson process we
now see that
\[ 
\begin{array}{rcl}
\Pee[ \Gamma \text{ is good} ] 
& =  & 
\left( \Pee[\Po( (n \pi r^2 / 4N^4 ) = 1 ] \right)^N
\cdot \Pee[\Po( n \cdot O( r^2 / N^2 ) ) = 0 ] \\
& = & 
\left( 
(n \pi r^2 / 4N^4) \exp (-n \pi r^2 / 4N^4) \right)^N \cdot \exp( - O( nr^2 / N^2 ) ) \\
& = & 
\exp \left( N \log(n\pi r^2/4) - 4N \log N - O( nr^2 / N^2 ) \right).
\end{array} 
\]
Considering the right hand side of the first inequality, the first term is the probability that the $N$ discs $B(c_i,\rho_2)$ contain exactly one random point, and the second term is  the probability that the $N$ sets
$(B(x_{i-1},r) \cap B(x_{i+1},r)) \backslash B(c_i, \rho_2)$ contain no random points.
We now choose $N = \lceil \left(n \pi r^2\right)^{1/4} \rceil$ and choose $K_3>0$ to be small enough so that we obtain
\[ 
\Pee( \Gamma \text{ is good} )  
\geq \exp \left( -O \left( \sqrt{nr^2} \right)  \right)
\geq \exp \left( - \frac12  \log n \right)  = n^{-\frac12}
\]
because  $r \leq K_3  \log n/ \sqrt{n}$ by assumption.
Also note that as promised before, the polygon $\Gamma$ fits easily inside the unit
square as it has  diameter $O( r N ) = O( r (n r^2)^{1/4} ) = o(1)$.
 
Let us now place shifted copies $\Gamma_1,\dots, \Gamma_M$ of $\Gamma$ inside the unit square
in such a way that they are contained in $[0,1]^2$ and their centers are separated by at least $10 \diam(\Gamma) =
\Theta( r N ) = \Theta( n^{1/4} r^{3/2} ) = n^{-1/2+o(1)}$.
(Recall we assumed without loss of generality that $r = \Omega(\sqrt{\log n/n})$.)
Then we can place $M = \Omega( (1/rN)^{2} ) = n^{1-o(1)}$ such shifted copies, with their centers forming a lattice in $[0,1]^2$.
Let $X$ denote the number of $\Gamma_i$s that are good.
Now notice that the events that the $\Gamma_i$ are good are independent of each other as they
concern disjoint areas of the plane.
Hence $X$ is distributed like a binomial with parameters $M = n^{1-o(1)}$ and $p \geq n^{-\frac12}$.
Thus:
\[ 
\Pee[ X = 0 ] = (1-p)^M \leq e^{-Mp} \leq e^{-n^{1/2-o(1)}} = o(1). 
\]
So $X > 0$ whp.

Consider the original random geometric graph $G_O$ again.
Let $X_P$ denote the number of good $\Gamma_i$s under the Poisson model, and let $X_O$ denote 
the number of good $\Gamma_i$s under the original model.
We have, with $K>0$ an arbitrary constant:
\begin{equation}
\label{eq:XOXP} 
\begin{array}{rcl}
\Pee[ X_O = 0 | X_P > 0 ] 
& = & 
\sum_{z=0}^\infty \Pee[ X_O = 0 | X_P > 0, Z=z ] \Pee[ X_P > 0 | Z=z ] \Pee[Z=z ] \\
& \leq & 
\sum_{z=0}^\infty \Pee[ X_O = 0 | X_P > 0, Z=z ] \Pee[ Z=z ] \\
& \leq & 
\sum_{z=n-K\sqrt{n}}^{n+K\sqrt{n}} \Pee[ X_O = 0 | X_P > 0, Z=z ] \Pee[ Z=z ] \\
& & + \Pee[ |Z-n| > K\sqrt{n} ].
\end{array}
\end{equation}
By Chebyschev's inequality we have
\[ 
\Pee \left[ |Z-n| > K\sqrt{n} \right] \leq \Var(Z) / (K\sqrt{n})^2 = 1/K^2. 
\]
Now consider the term $\Pee[ X_O = 0 | X_P > 0, Z=z ]$.
If $z=n$ then it clearly equals 0.
Let us take $n-K\sqrt{n} \leq z < n$.
If we condition on the event that $X_P > 0, Z=z$, then we
can fix a good $\Gamma_i$, say with ``corners'' $(x_{i_1},\dots,x_{i_N})$.
If $X_O = 0$ then the set 
$A := \bigcup_{j=1}^N W(x_{i_{j-1}},x_{i_{j+1}})$ must contain
one of the points $x_{z+1},\dots,x_n$.
By equation \eqref{eqn:area-lower-bound}, $\area(A) = N \cdot O( r^2/N^3 ) = O(r^2/N^2)$.
Thus, for $n-K\sqrt{n} < z < n$ we have
\[ 
\begin{array}{rcl}
\Pee[ X_O=0|X_P>0,Z=z] 
& \leq & 
(n-z) \cdot O(r^2/N^2) \\
& \leq & 
K \sqrt{n} \cdot O( r/\sqrt{n} ) \\
& = & 
o(1), 
\end{array}  
\]
using $N = \lceil \left( \pi nr^2\right)^{1/4} \rceil$.
Observe that the $o(1)$ bound is {\em uniform} over all $n-K\sqrt{n} < z < n$.

Similarly, if we condition on the event that $X_P > 0, Z=z$ with $n < z \leq n+K\sqrt{n}$, we
can pick an $N$-tuple $(x_{i_1},\dots,x_{i_N})$ uniformly at random from all $N$-tuples that are
``corners'' of a good $\Gamma_i$.
The indices $i_1,\dots,i_N$ are a uniformly random sample (without replacement)
from $\{1,\dots,z\}$.
Now, if $X_O = 0$, it must hold that one of $i_1,\dots i_N$ is larger than $n$.
Note $\Pee( i_j > n ) = (z-n)/z$ for $j=1,\dots,N$, and so
\[ 
\Pee[ X_O=0|X_P>0,Z=z]  \leq  
N \left(\frac{z-n}{z}\right) \\ 
 \leq  
(\pi nr^2)^{\frac14} \left(\frac{K\sqrt{n}}{n}\right) \\
 =  
K \pi^{\frac{1}{4}} n^{-\frac14} r^{\frac12} \\
 =  
o(1). 
\]
Observe that again the $o(1)$ bound is {\em uniform} over all $z$ considered.
Combining these bounds with~\eqref{eq:XOXP} we get
\[ 
\begin{array}{rcl}
\Pee[ X_O = 0 | X_P > 0 ] 
& \leq &  
1/K^2 + \sum_{z=n-K\sqrt{n}}^{n+K\sqrt{n}} o(1) \cdot \Pee[ Z=z ] \\
& = & 
1/K^2 + o(1). 
\end{array} 
\]
By sending $K\to\infty$, we see that $\Pee[ X_O = 0 | X_P > 0 ] = o(1)$, so 
\[ 
\Pee[ X_O > 0 ] \geq  
\Pee[ X_O > 0 | X_P > 0 ] \Pee[ X_P > 0 ] = (1-o(1))(1-o(1)) = 1-o(1), 
\]
which concludes the proof.
\end{proofof}

\bibliography{geocop-bib}

\providecommand{\bysame}{\leavevmode\hbox to3em{\hrulefill}\thinspace}
\providecommand{\MR}{\relax\ifhmode\unskip\space\fi MR }
\providecommand{\MRhref}[2]{%
  \href{http://www.ams.org/mathscinet-getitem?mr=#1}{#2}
}
\providecommand{\href}[2]{#2}
\begin{thebibliography}{10}

\bibitem{aigner+fromme}
M.~Aigner and M.~Fromme, \emph{A game of cops and robbers}, Discrete Appl.
  Math. \textbf{8} (1983), 1--12.

\bibitem{alon}
N.~Alon, personal communication, 2011.

\bibitem{alspach}
B.~Alspach, \emph{Sweeping and searching in graphs: a brief survey},
  Mathematiche \textbf{59} (2006), 5--37.

\bibitem{avin+ercal}
C.~Avin and G.~Ercal, \emph{On the cover time and mixing time of random
  geometric graphs}, Theor. Comput. Sci. \textbf{380} (2007), no.~1-2, 2--22.

\bibitem{baird+bonato}
W.~Baird and A.~Bonato, \emph{{M}eyniel's conjecture on the cop number: a
  survey}, Submitted.

\bibitem{bhadauria+isler}
D.~Bhadauria and V.~Isler, \emph{Capturing an evader in a polygonal environment
  with obstacles}, 22nd International Joint Conference on Artificial
  Intelligence, 2011, To appear.

\bibitem{bollobas+kun}
B.~Bollob\'{a}s, G.~Kun, and I.~Leader, \emph{Cops and robbers in random
  graphs}, Submitted.

\bibitem{bonato+kemes+pralat}
A.~Bonato, G.~Kemkes, and P.~Pra{\l}at, \emph{Almost all cop-win graphs contain
  a universal vertex}, Submitted.

\bibitem{bonato+nowakowski}
A.~Bonato and R.~Nowakowski, \emph{The game of cops and robbers on graphs},
  American Mathematical Society, 2011.

\bibitem{bonato+pralat+wang}
A.~Bonato, P.~Pra{\l}at, and C.~Wang, \emph{Pursuit-evasion in models of
  complex networks}, Internet Mathematics \textbf{4} (2007), no.~4, 419--436.

\bibitem{cooper+frieze}
C.~Cooper and A.~Frieze, \emph{The cover time of random geometric graphs},
  Random Structures and Algorithms \textbf{38} (2011), 324--349.

\bibitem{frankl2}
P.~Frankl, \emph{Cops and robbers in graphs with large girth and {C}ayley
  graphs}, Discrete Applied Mathematics \textbf{17} (1987), 301--305.

\bibitem{frankl}
\bysame, \emph{On a pursuit game on {C}ayley graphs}, Combinatorica \textbf{7}
  (1987), no.~1, 67--70.

\bibitem{FKL}
A.~Frieze, M.~Krivelevich, and P.-S. Lo, \emph{Variations on cops and robbers},
  J. Graph Theory, To appear.

\bibitem{guibas}
L.~J. Guibas, J.-C. Latombe, S.~M. Lavalle, D.~Lin, and R.~Motwani, \emph{A
  visibility-based pursuit-evasion problem}, International Journal of
  Computational Geometry and Applications \textbf{9} (1996), 471--494.

\bibitem{gupta+kumar}
P.~Gupta and P.~R. Kumar, \emph{Critical power for asymptotic connectivity in
  wireless netowrks}, Stochastic Analysis, Control, Optimization and
  Applications (Boston), Birkha\"{u}ser, 1998.

\bibitem{hahn}
G.~Hahn, \emph{Cops, robbers and graphs}, Tatra Mountain Mathematical
  Publications \textbf{36} (2007), 163--176.

\bibitem{isler05tro}
V.~Isler, S.~Kannan, and S.~Khanna, \emph{Randomized pursuit-evasion in a
  polygonal environment}, IEEE Transactions on Robotics \textbf{5} (2005),
  no.~21, 864--875.

\bibitem{kingmanboek}
J.~Kingman, \emph{Poisson processes}, Oxford University Press, Oxford, 1993.

\bibitem{KR}
S.~Kopparty and C.~V. Ravishankar, \emph{A framework for pursuit evasion games
  in $\mathbb{R}^n$}, Inf. Process. Lett. \textbf{96} (2005), 114--122.

\bibitem{lavalle}
S.~M. Lavalle, D.~Lin, L.~J. Guibas, J.-C. Latombe, and R.~Motwani,
  \emph{Finding an unpredictable target in a workspace with obstacles},
  Proceedings of the 1997 IEEE International Conference on Robotics and
  Automation, 1997, pp.~737--742.

\bibitem{lu+peng}
L.~Lu and X.~Peng, \emph{On {M}eyniel's conjecture of the cop number},
  Submitted.

\bibitem{luczak+pralat}
T.~{\L}uczak and P.~Pra{\l}at, \emph{Chasing robbers on random gaphs: zigzag
  theorem}, Random Structures and Algorithms \textbf{37} (2010), 516--524.

\bibitem{mehrabian}
A.~Mehrabian, \emph{The capture time of grids}, Discrete Mathematics
  \textbf{311} (2011), no.~1, 102 -- 105.

\bibitem{neufeld}
S.~Neufeld and R.~Nowakowski, \emph{A game of cops and robbers played on
  products of graphs}, Discrete Mathematics \textbf{186} (1998), no.~1-3, 253
  -- 268.

\bibitem{Nowakowski+Winkler}
R.~Nowakowski and P.~Winkler, \emph{Vertex-to-vertex pursuit in a graph},
  Discrete Mathematics \textbf{43} (1983), no.~2--3, 235 -- 239.

\bibitem{penroseMST}
M.~D. Penrose, \emph{The longest edge of the random minimal spanning tree},
  Ann. Appl. Probab. \textbf{7} (1997), no.~2, 340--361.

\bibitem{penrose}
M.~D. Penrose, \emph{Random geometric graphs}, Oxford University Press, 2003.

\bibitem{pralat}
P.~Pra{\l}at, \emph{When does a random graph have constant cop number?},
  Australasian Journal of Combinatorics \textbf{46} (2010), 285--296.

\bibitem{pralat+wormald}
P.~Pra{\l}at and N.~Wormald, \emph{{M}eyniel's conjecture holds in random
  graphs}, Submitted.

\bibitem{quilliot}
A.~Quilliot, \emph{Jeux et pointes fixes sur les graphes}, Ph.D. thesis,
  Universit\'e de Paris VI, 1978.

\bibitem{scott+sudakov}
A.~Scott and B.~Sudakov, \emph{A bound for the cops and robbers problem}, SIAM
  J. Discrete Math., To appear.

\bibitem{sgall}
J.~Sgall, \emph{A solution to {D}avid {G}ale's lion and man problem},
  Theoretical Comp. Sci. \textbf{259} (2001), no.~1--2, 663--670.

\bibitem{xue+kumar}
F.~Xue and P.~R. Kumar, \emph{Scaling laws for ad hoc wireless networks: an
  information theoretic approach}, Foundations and Trends in Networking
  \textbf{1} (2006), no.~2, 145--270.

\end{thebibliography}

\end{document}